\documentclass[10pt,twoside]{article}

\usepackage{amssymb,amsmath}
\usepackage[T1]{fontenc}
\usepackage{xfrac,faktor}

\newcommand{\al}{\alpha}

\newcommand{\be}{\beta}

\newcommand{\ga}{\gamma}

\newcommand{\epsn}{\varepsilon_0}

\newcommand{\la}{\lambda}

\newcommand{\om}{\omega}

\newcommand{\tht}{\vartheta}

\renewcommand{\phi}{\varphi}

\newcommand{\N}{{\mathbb N}}

\newcommand{\Lim}{\mathrm{Lim}}

\newcommand{\thtn}{\tht_n}
\newcommand{\thtnod}{\tht_0}

\newcommand{\T}{{\operatorname{T}}}

\newcommand{\mc}{{\operatorname{mc}}}

\newcommand{\PA}{{\operatorname{PA}}}

\newcommand{\pioneonecanod}{\Pi^1_1{\operatorname{-CA}_0}}

\newcommand{\IDone}{{\operatorname{ID}_1}}

\newcommand{\qed}{\mbox{ }\hfill $\Box$\vspace{2ex}}

\newlength{\hilflh}

\newtheorem{theo}{Theorem}[section]
\newtheorem{cor}[theo]{Corollary}
\newtheorem{lem}[theo]{Lemma}
\newtheorem{defi}[theo]{Definition}

\newcommand{\Romannumeral}[1]{\uppercase\expandafter{\romannumeral #1\relax}}

\newcommand{\tauCkquok}{{\sfrac{{\tau}}{{_\Ck}}}}
\newcommand{\tauCkquoke}{{\sfrac{{\tau}}{{_\Cke}}}}
\newcommand{\tauCkquol}{{\sfrac{{\tau}}{{_\Cl}}}}

\newcommand{\Ce}{{\mathrm{G}_1}}

\newcommand{\Ck}{{\mathrm{G}_k}}
\newcommand{\Ckstar}{{\mathrm{G}^*_k}}

\newcommand{\Cke}{{\mathrm{G}_{k+1}}}
\newcommand{\Ckz}{{\mathrm{G}_{k+2}}}
\newcommand{\Ckl}{{\mathrm{G}_{k+l}}}
\newcommand{\Cklstar}{{\mathrm{G}^*_{k+l}}}
\newcommand{\Cl}{{\mathrm{G}_l}}

\newcommand{\Pj}{{\operatorname{P}_j}}
\newcommand{\Pk}{{\operatorname{P}_k}}

\newcommand{\Pkjtimes}{{\operatorname{P}^{(j)}_k}}
\newcommand{\Pke}{{\operatorname{P}_{k+1}}}
\newcommand{\Pkel}{{\operatorname{P}_{k+1+l}}}
\newcommand{\Pkz}{{\operatorname{P}_{k+2}}}
\newcommand{\Pkl}{{\operatorname{P}_{k+l}}}

\newcommand{\hk}{{\mathrm{h}_k}}
\newcommand{\hke}{{\mathrm{h}_{k+1}}}

\newcommand{\imc}{{\mathrm{imc}}}

\newcommand{\istarkal}{{\mathrm{i}^\star_{k,\al}}}

\begin{document}

\title{Uniform generalization of Goodstein's theorem and Cichon's independence proof\footnote{This article is available on arXiv, see \cite{W_this_on_arXiv}.}}

\author{
Gunnar Wilken\footnote{ORCID iD: 0000-0002-4019-9320, E-mail: wilken@oist.jp}\\
Okinawa Institute of Science and Technology, Japan\\
}

\maketitle

\begin{abstract} We uniformly generalize Goodstein's theorem \cite{Goodstein1944} and Cichon's independence proof \cite{Cichon1983} 
via the concept of maximality quotients.
\end{abstract}

\section{Introduction}

This article results from a careful inspection of the article \cite{Cichon1983} by Cichon, which shows termination and independence of Goodstein's
famous theorem \cite{Goodstein1944} by a reduction of the provability of Goodstein's theorem to the provability of totality of the Hardy function $H_{\epsn}$. 
According to Kreisel, see \cite{Kreisel1952}, Peano Arithmetic ($\PA$) does not prove the totality of $H_{\epsn}$, a result that in turn builds upon Gentzen's 
proof of consistency of $\PA$ by an ordinal analysis that uses transfinite induction up to the proof-theoretic ordinal $\epsn$ of $\PA$, see \cite{Gentzen1936,Gentzen1938}.

Cichon's proof is specific for $\PA$ and its proof-theoretic ordinal $\epsn$ in that it uses
the fact that writing a natural number $N$ hereditarily as sum of powers of a base $k\ge2$, by replacing all occurrences of $k$ in the term for $N$ by $\om$, 
one obtains an ordinal $\al<\epsn$ in Cantor normal form that satisfies $\Ck(\al)=N$, where $\Ck$ is a collapsing function defined in subsection 
\ref{CkPksubsec}, while on the other hand any $\al<\epsn$ in Cantor normal form (i.e.\ $\al$ is written hereditarily as a sum of powers of $\om$) 
translates back via replacing $\om$ by $k$ to a term that evaluates to $\Ck(\al)$, cf.\ Lemma 1 and Remarks 1 and 2 of \cite{Cichon1983}.

In this article we uniformly generalize Goodstein's theorem and Cichon's independence proof by exploiting the fact that for a suitable notation
system for an ordinal $\tau$ described in section \ref{tricksec}, each collapsing function $\Ck$ for $k\ge2$ induces an equivalence relation on the ordinals 
below $\tau$, each class of which contains a maximum, which we use as canonical representative of that equivalence class. Thus, the quotients $\tauCkquok$
induced by $\Ck$ can be identified with the sets of maximal representatives, which is why we call such quotients {\it maximality quotients}, 
and the restriction of $\Ck$ to $\tauCkquok$ then results in an order-isomorphism with $\N$. Given $N\in\N$, we may therefore define the base-$k$ 
representation of $N$ by the unique preimage within $\tauCkquok$ under $\Ck$. This then gives rise to a uniform generalized Goodstein principle 
and related independence result. 

Recent progress related to Goodstein principles was achieved by \cite{AFDWW20} and \cite{AWW21}, elaborating on Arai's idea of normal form representations for 
natural numbers using the Ackermann function (and variants) and also building on earlier ideas by Weiermann in \cite{We17}. Maximality of ordinal indices used in the definition
of normal forms for natural numbers that make use of fast growing functions like (extended variants of) the Ackermann function played a crucial role there as well, and that approach 
was further extended by Fern\'andez-Duque and Weiermann in \cite{FDW24b} and \cite{FDW25} to a broader variety of Goodstein principles of strength up to the Bachmann-Howard
ordinal, the proof theoretic ordinal of the theory $\IDone$ of non-iterated inductive definitions, for which they could prove termination 
and independence uniformly via the notion of base-change maximality of normal form representations of natural numbers. 

The present article emerged from \cite{W}, which started out from a different type of quotient, called $\imc$-quotient, that was motivated by the measure 
$\mc(\al)$, the maximal coefficient of $\al$, for ordinals $\al<\epsn$. $\mc(\al)$ counts the maximum number of times the same additively indecomposable ordinal 
is consecutively added in the Cantor normal form of $\al$ and plays a prominent role in \cite{FDW24b}. In subsection 2.6 of \cite{W} we introduced a measure $\imc(\al)$,
the iterative maximal coefficient of $\al$, where $\al$ is an ordinal in the notation system $\T$ characterizing the proof-theoretic ordinal of the subsystem 
$\pioneonecanod$ of second order number theory, which is the strongest among the so-called big five of reverse mathematics. 
A Goodstein principle independent of $\pioneonecanod$ had already been achieved by \cite{WW13} using tree ordinals and term rewriting, which may be seen to be somewhat less
natural.

Closer inspection of \cite{W}, however, revealed the previously overlooked uniformity of the argument that only involves maximality quotients and is therefore independent of a 
concrete notation system. An updated version of \cite{W} now shows the equality of the maximality quotients to their $\imc$-quotient counterparts for the notation system $\T$,
see section 3 of \cite{W}.
We believe that the characterization of maximality quotients by $\imc$-quotients provides deeper insight into concrete notation systems such as $\T$ and will prove 
useful elsewhere. 

A uniform approach to fundamental sequences that satisfy the Bachmann property (see property 4 at the beginning of section \ref{tricksec}) was given by 
Buchholz, Cichon, and Weiermann in \cite{BCW94}.
The approach of \cite{BCW94} yields systems of fundamental sequences satisfying the Bachmann property, with no upper bound on the size of the (proof-theoretic) 
ordinals considered. Note that $k$-maximality for small $k$ naturally corresponds to a low norm value in the context of \cite{BCW94}.
Such systems are therefore natural candidates for the framework developed here. We consider our additional assumption \ref{masterseqprop} made at the
beginning of section \ref{tricksec} to be quite natural for notation systems, cf.\ the examples in section \ref{tricksec}.

\section{A uniform framework for Cichon's trick}\label{tricksec}

Let $(\tau,\cdot[\cdot])$ be a system of ordinal notations with assigned fundamental sequences that satisfy the Bachmann property,
where $\tau$ is a countable (recursive) limit ordinal, e.g.\ the proof-theoretic ordinal of a ($\Pi^1_1$-sound) theory $\mathcal{T}$ for which a proper 
($\Pi^0_2$-) ordinal analysis and the unprovability of totality of $H_\tau$ within $\mathcal{T}$ are available, where $H_\tau$ is defined with respect to 
a ``master'' fundamental sequence $(\tau_n)_{n<\om}$ described below. 
For the system of fundamental sequences we assume the following to hold:
\begin{enumerate}
\item $0[n]=1[n]=0$ for any $n<\om$,
\item $(\xi+\eta)[n]=\xi+\eta[n]$ for any $\xi<\tau$, $\eta\in(0,\tau)$, and $n<\om$,
\item for any limit ordinal $\la<\tau$, the sequence $(\la[n])_{n<\om}$, where $\la[0]>0$, is strictly increasing with supremum $\la$,
\item \emph{Bachmann property:} For any $\al,\be<\tau$ and $n<\om$, if $\al[n]<\be<\al$, then $\al[n]\le\be[0]$, and
\item\label{masterseqprop} for the (strictly increasing) master fundamental sequence, in short: \emph{master sequence}, $(\tau_n)_{n<\om}$, 
we have $\sup\{\tau_n\mid n<\om\}=\tau$ and $\tau_{n+1}[0]=\tau_n$ for all $n<\om$, and we also write $\tau[n]$ for $\tau_n$, slightly abusing notation. 
\end{enumerate}
Such master sequences as in assumption \ref{masterseqprop} are quite canonical in classical notation systems. 
This also pertains to the required property $\tau_{n+1}[0]=\tau_n$ for all $n$, which crucially applies in the proof of Lemma \ref{masterseqmaxlem}.

\medskip

\noindent {\it Examples for master sequences:} 
\begin{enumerate}
\item In the case of $\epsilon_0$, the proof-theoretical ordinal of Peano Arithmetic $\PA$, the classical system of fundamental sequences
\[0[k]:=0,\;\;1[k]:=0,\;\;(\xi+\eta)[k]:=\xi+\eta[k],\;\;\om^{\be+1}[k]:=\om^\be\cdot(k+1),\;\mbox{ and }\;\om^\la[k]:=\om^{\la[k]}\]
satisfies properties $1$ - $4$, and the towers of $\om$-exponentiation, given by $\om_0:=1$ and $\om_{n+1}=\om^{\om_n}$ for $n<\om$, 
satisfy $\om_{n+1}[0]=\om_n$ for all $n<\om$ and $\sup\{\om_n\mid n<\om\}=\epsn$. 
\item For the notation system $\T$ of strength $\pioneonecanod$ considered in \cite{W26,W}, we introduced a system of fundamental sequences in \cite{W26} that satisfies 
the Bachmann property as well as properties $1$ - $3$. Its master sequence is $(\thtnod(\ldots(\thtn(0))\ldots))_{n<\om}$, which is the sequence starting
with $\om,\epsilon_0,$ and then enumerating the proof-theoretic ordinals of the systems $\mathrm{ID}_n$, $0<n<\om$, of $n$-times iterated inductive definitions. 
\end{enumerate}

\subsection{Slow growing hierarchy and predecessor operations}\label{CkPksubsec}
The slow-growing hierarchy written as collapsing function $\Ck:\tau\to\om$ is defined for $k>0$ by\footnote{Induction on $\al$ shows that our definitions are 
equivalent to Cichon's in \cite{Cichon1983} for $k>0$, where Cichon's $\al[k]$ is equal to our $\al[k-1]$.} 
\begin{equation}\label{slowgrowingeq} \Ck(0):=0,\: \Ck(\al+1):=\Ck(\al)+1, \mbox{ and }\Ck(\la):=\Ck(\la[k-1])\mbox{ for }\la\in\Lim,\end{equation} 
where $\Lim$ denotes the class of limit ordinals, and the predecessor operations $\Pk:\tau\to\tau$ for $k>0$ are defined equivalently as in \cite{Cichon1983} by
\begin{equation}\label{predeq} \Pk(0):=0,\: \Pk(\al+1):=\al, \mbox{ and }\Pk(\la):=\Pk(\la[k-1])\mbox{ for }\la\in\Lim.\end{equation}

Note first of all that Cichon's crucial lemma also holds in our context, with the same proof by induction on $\al$:
\begin{lem}[cf.\ Lemma 2 of \cite{Cichon1983}]\label{tricklem} For $k>0$ and $\al<\tau$ the operations $\Ck$ and $\Pk$ commute, thus \[\Ck\Pk\al=\Pk\Ck\al,\]
where the latter is equal to $\Ck\al-1$. \qed
\end{lem} 

We denote iterations of descendance along fundamental sequences as follows: \[\al[n]^0:=\al\quad\mbox{ and }\quad\al[n]^{i+1}:=\al[n]^i[n].\]
For $k>0$ let $\istarkal$ be the minimal $i$ such that $\al[k-1]^i=0$.
\begin{lem}\label{Ckvaluelem} Let $k>0$ and $\al<\tau$.
\begin{enumerate}
\item The sequence $(\Ck(\al[k-1]^i))_{i\le\istarkal}$ is weakly decreasing and passing through each $m\le\Ck(\al)$, 
such that for $i<\istarkal$ we have
\[\Ck(\al[k-1]^{i+1})<\Ck(\al[k-1]^i)\;\Longleftrightarrow\;\al[k-1]^i\mbox{  is a successor ordinal.}\] 
\item Let $i_1,\ldots,i_n$ be the increasing enumeration of those $i<\istarkal$ that satisfy $\Ck(\al[k-1]^{i+1})<\Ck(\al[k-1]^i)$. 
We have $n=\Ck(\al)$ by definition, $n>0$ if and only if $\al>0$ according to our assumptions regarding $(\tau,\cdot[\cdot])$, and 
for $j\in\{1,\ldots,n\}$ we have
\[\Pkjtimes(\al)=\al[k-1]^{i_j+1}.\]
\end{enumerate}
\end{lem}
{\bf Proof.} This is an immediate consequence of the fact that $(\be+1)[l]=\be$ for any $\be<\tau$ and $l<\om$, and of the definitions of $\Ck$ and $\Pk$ and 
Lemma \ref{tricklem}: Note that both $\Ck$ and $\Pk$ recur to $\la[k-1]$ at limits $\la$  and observe that for $\al>0$ the sequence $(\Pkjtimes(\al))_{j\le n}$ 
results from the sequence $(\al[k-1]^i)_{i\le\istarkal}$ by omitting exactly those elements $\al[k-1]^{i+1}$ (where $i<\istarkal$) such that both $\al[k-1]^{i+1}$ and $\al[k-1]^i$ are limits. 

Note that in part 2 we have $n>0$ if and only if $\al>0$ due to our assumption that for limits $\la<\tau$ we have $\la[0]>0$. This assumption, which we consider to be
more natural than allowing $\la[0]=0$ for any limit $\la$, enables part 2 to hold for all $k>0$. By definition of $\Ck$ and $\Pk$ we have $\Ck(\al)=\min\{j\mid\Pkjtimes(\al)=0\}$,
and by definition of $\istarkal$ we have $\istarkal=i_n+1$.
\qed

\subsection{Maximality quotients}\label{maxquotientssubsec}

\begin{lem}\label{Bachmannintervallem}
For any $j<\om$,  $k\in[2,\om)$, and $\al,\be<\tau$ such that $\al\in(\be[j],\be)$ we have \[\Ck(\al)>\Ck(\be[j])\] and $\be[j]=\al[k-1]^{h+1}$ for some $h$ where $\al[k-1]^h$ is a successor.
\end{lem}
{\bf Proof.} Note first, that if $\al\in\Lim$, due to the Bachmann property we have $\be[j]\le\al[0]<\al[k-1]$ since by assumption $k\ge2$.\footnote{Here we essentially need the 
assumption $k\ge2$. For $k=1$ we would have $\al[0]=\al[k-1]$ and could easily find counterexamples: For instance, consider $\tau=\epsn$ with the system of 
fundamental sequences specified in the first example at the beginning of section \ref{tricksec}. We have $\om^2\in(\om,\om^\om)$ and $\om^\om[0]=\om$, while
$\Ce(\om^2)=\Ce(\om)$ since $\om^2[0]=\om$.}  
Now, for any $\al$ according to the lemma's assumptions, descending from $\al$ can only reach $\be[j]$ after passing through successor stages: for the least $h$ such that 
$\al[k-1]^{h+1}\le\be[j]$, which exists since the iterated application of $\cdot[k-1]$ yields a strictly decreasing sequence until eventually reaching $0$, the ordinal $\al[k-1]^h$ 
must be a successor, whence $\al[k-1]^{h+1}=\be[j]$. Thus $\Ck(\al)>\Ck(\be[j])$ by Lemma \ref{Ckvaluelem}. 
\qed

\begin{lem}\label{masterseqmaxlem}
For any $k\ge2$ the sequence $(\Ck(\tau_n))_{n<\om}$ is strictly increasing, and each $\tau_n$ is maximal with respect to $\Ck$.
\end{lem}
{\bf Proof.} Due to the Bachmann property the sequence $(\tau_{n+1}[k-1]^i)_{i<\om}$ passes through $\tau_{n+1}[0]=\tau_n$ after hitting successor stages, 
cf.\ the proof of Lemma \ref{Bachmannintervallem}, hence $\Ck(\tau_{n+1})>\Ck(\tau_n)$. Thus, the sequence $(\Ck(\tau_n))_{n<\om}$ is strictly increasing. 

Moreover, for any $\be\in(\tau_j,\tau_{j+1})$, as $\tau_{j+1}[0]=\tau_j$, Lemma \ref{Bachmannintervallem} yields $\Ck(\be)>\Ck(\tau_j)$.
From this we conclude that $\Ck(\be)>\Ck(\tau_n)$ holds for all $\be>\tau_n$. Indeed, picking $j\ge n$ minimally such that $\be<\tau_{j+1}$, we have $\tau_j\le\be<\tau_{j+1}$.
The situation $\be=\tau_j$ implies $j>n$ since $\be>\tau_n$, and is covered by the first paragraph; and if $\be>\tau_j$, we obtain $\Ck(\be)>\Ck(\tau_j)\ge\Ck(\tau_n)$. 
Hence each $\tau_n$ is maximal. 
\qed

Even though we do not need the following lemma downstream, we consider it worth spelling out. So far, by Lemma \ref{masterseqmaxlem} we have the maximality with respect
to $\Ck$ of the elements $\tau_n$ of the master sequence, even shown in the strong form that for all $\be>\tau_n$ we have $\Ck(\be)>\Ck(\tau_n)$. We now show that maximality
always implies this strong form of maximality, cf.\ Corollary 2.33 of \cite{W} in the context of $\imc$-quotients.
\begin{lem}\label{strongmaxlem} Let $k\ge2$ and suppose $\al<\tau$ to be maximal among all $\ga<\tau$ such that $\Ck(\ga)=\Ck(\al)$, i.e.\ {\it maximal with respect to $\Ck$}. 
Then for all $\be\in(\al,\tau)$ we have $\Ck(\be)>\Ck(\al)$.
\end{lem} 
{\bf Proof.} Assume to the contrary that there existed a least $\be>\al$ such $\Ck(\be)\le\Ck(\al)$. As the sequence $(\be[k-1]^j)_{j<\om}$ strictly decreases until reaching $0$, 
there exists a unique $i$ such that $\be[k-1]^{i+1}\le\al<\be[k-1]^i$. 
We can not have $\al=\be[k-1]^{i+1}$: in case of $\be[k-1]^i\in\Lim$ we would have $\Ck(\al)=\Ck(\be[k-1]^i)$, contradicting the maximality of $\al$, and if
$\be[k-1]^i=\be[k-1]^{i+1}+1$, we would have $\Ck(\al)<\Ck(\be[k-1]^i)\le\Ck(\be)$ by Lemma \ref{Ckvaluelem}, contradicting our assumption.
Thus $\al\in(\be[k-1]^{i+1},\be[k-1]^i)$, whence $\be[k-1]^i\in\Lim$ and $\Ck(\be[k-1]^{i+1})=\Ck(\be[k-1]^i)$. 
As a consequence of the Bachmann property, according to Lemma \ref{Bachmannintervallem} we would then have $\Ck(\al)>\Ck(\be[k-1]^i)$, 
which in case of $i>0$ contradicts the minimality of $\be$. We would thus have $\al\in(\be[k-1],\be)$, and setting $l:=\Ck(\al)-\Ck(\be)$ we would obtain
$\Ck(\be+l)=\Ck(\be)+l=\Ck(\al)$ while $\be+l>\al$, contradicting the maximality of $\al$. 
\qed
 
 \begin{lem}\label{Pkiterationmaxlem}
 For $k\ge2$ and $n<\om$ set $\al:=\tau_n$ and $m_i:=\Ck(\al[k-1]^i)$ for $i\le\istarkal$. We then have
 \begin{enumerate}
 \item $\Ck(\be)\ge m_i$ for all $\be\ge\al[k-1]^i$ and $i\le\istarkal$.
 \item $\Pkjtimes(\al)$ is maximal with respect to $\Ck$ for all $j\le\Ck(\al)$. 
 \item $\Ck(\Pkjtimes(\al))=\Ck(\al)-j$ for $j\le\Ck(\al)$.
 \item For each $n_0\le n$ there exists a $j\le\Ck(\al)$ such that $\tau_{n_0}=\Pkjtimes(\al)$.
  \end{enumerate}
 \end{lem}
 {\bf Proof.} Part 1 is shown by straightforward induction on $i\le\istarkal$, using Lemmas \ref{Ckvaluelem}, \ref{Bachmannintervallem}, 
 and for $i=0$ (the proof of) Lemma \ref{masterseqmaxlem}.
 
 Part 2 is a consequence of part 1. 
 Indeed, let $i_1,\ldots,i_{\Ck(\al)}$ be the sequence of indices according to part 2 of Lemma \ref{Ckvaluelem}, so that $\Pkjtimes(\al)=\al[k-1]^{i_j+1}=\al[k-1]^{i_j}-1$ for $j=1,\ldots,\Ck(\al)$,
 and for all $j\le\Ck(\al)$ and $\be>\Pkjtimes(\al)$ we have $\Ck(\be)>\Ck(\Pkjtimes(\al))$.
 
 Part 3 follows from the definitions of $\Ck$ and $\Pk$ via an induction on $j$, applying Lemma \ref{tricklem}.
 
 Part 4 was pointed out at the beginning of the proof of Lemma \ref{masterseqmaxlem}, and the general argumentation was given in (the proof of) Lemma \ref{Bachmannintervallem}.
\qed
 
\begin{cor}\label{maxexistencepkmaxprescor}
Let $k\ge2$. 
\begin{enumerate}
\item For any $m<\om$ there exists a maximal $\al<\tau$ such that $\Ck(\al)=m$. 
\item The predecessor operation $\Pk$ preserves maximality with respect to $\Ck$.
\end{enumerate}
\end{cor}
{\bf Proof.}
Choosing $n$ such that $\Ck(\tau_n)>m$, which is possible due to Lemma \ref{masterseqmaxlem}, and descending from $\tau_n$ by iteration of the operation $\Pk$, 
according to Lemma \ref{Pkiterationmaxlem} we eventually reach $\al:=\Pkjtimes(\tau_n)$, which is maximal with respect to $\Ck$,  such that $\Ck(\Pkjtimes(\al))=m$. 
This shows part 1.

For part 2, Lemma \ref{Pkiterationmaxlem} yields that any $\al$ that is maximal with respect to $\Ck$ is of a form $\Pkjtimes(\tau_n)$, and that further applications of $\Pk$ maintain
maximality. 
\qed

We come to the central notion in this article, the notion of {\it maximality quotient} of a set of ordinal terms. Part 1 of Corollary \ref{maxexistencepkmaxprescor} gives rise to the definition 
of a quotient $\tauCkquok$ with respect to the equivalence relation induced by $\Ck$, i.e.\ $\al,\be<\tau$ are equivalent if and only if $\Ck(\al)=\Ck(\be)$, in which elements
can be identified with their maximal representative. 

\begin{defi}[Maximality quotients] \label{maximalityquotientdefi}For $k\ge2$ we define
\[\tauCkquok:=\{\max\left\{\al<\tau\mid\Ck(\al)=m\}\mid m<\om\right\}.\]
\end{defi}

\begin{cor}\label{maxquotclscor} For $k\ge2$ the maximality quotient $\tauCkquok$ is the closure of $\{\tau_n\mid n<\om\}$ under the operation $\Pk$. 
\end{cor}
{\bf Proof.} Immediate from Lemmas \ref{masterseqmaxlem}, \ref{Pkiterationmaxlem}, and Corollary \ref{maxexistencepkmaxprescor}.
\qed

\begin{lem}\label{maxquotmonotolem} The sequence of maximality quotients $(\tauCkquok)_{k\ge2}$ is $\subseteq$-increasing.
\end{lem}
{\bf Proof.} We have $\{\tau_n\mid n<\om\}\subseteq\tauCkquok$ for every $k\ge2$ according to Lemma \ref{masterseqmaxlem}. 
By Corollary \ref{maxquotclscor} any $\al\in\tauCkquok$ is the result of iterated application of $\Pk$ to some $\tau_n$.
But, again as a consequence of Lemma \ref{Bachmannintervallem}, the sequence built by iterated applications of $\cdot[k-1]$ to $\tau_n$ is a subsequence of the sequence built by 
iterations of $\cdot[k]$ starting from $\tau_n$. Since $\Pk$ and $\Pke$ act identically on successors (by subtracting $1$), we infer that therefore also the sequence built by iterated 
applications of $\Pk$ to $\tau_n$ is a subsequence of the sequence built by iterations of $\Pke$ starting from $\tau_n$. 
Thus, we also have $\al\in\tauCkquoke$.
\qed

\begin{theo}\label{maxquotisomtheo}
The restriction $\Ckstar$ of $\Ck$ to $\tauCkquok$ \[\Ckstar:\quad\tauCkquok\to\N\]
is an order isomorphism, and the restriction of $\Pk$ to $\tauCkquok$ is the true predecessor function on $\tauCkquok$.
\end{theo}
{\bf Proof.} $\Ckstar$ is injective by definition and surjective by Corollary \ref{maxexistencepkmaxprescor}.
As a consequence of Lemma \ref{Pkiterationmaxlem}, the sequences resulting from iterating $\Pk$ on the ordinals $\tau_n$ and writing them in increasing order results in 
end-extensions by finitely many elements in each step of incrementation of $n$. Corollary \ref{maxquotclscor} shows that this process
exhausts all of $\tauCkquok$. According to Lemma \ref{tricklem}, the application of $\Ck$ to the resulting increasing enumeration of  $\tauCkquok$ then enumerates $\N$ in
increasing order, cf.\ part 3 of Lemma \ref{Pkiterationmaxlem}.
\qed

\section{Goodstein process}\label{Goodsteinprocesssec}

We now define the abstract notion of base-$k$ representation of natural numbers and the uniform Goodstein process for $\mathcal{T}$ with respect to $(\tau,\cdot[\cdot])$. 
\begin{defi}\label{basechangedefi} Let $N\in\N$ and $k\ge2$. According to Theorem \ref{maxquotisomtheo} there exists a unique $\al\in\tauCkquok$ such that $N=\Ck(\al)$, 
the base-$k$ representation of $N$. 
For any $l\in(k,\om)$ the change of base from $k$ to $l$ for $N$ is defined by \[N[k\mapsto l]:=\Cl(\al),\]
so that $\al$ is the base-$l$ representation of $N[k\mapsto l]$ since $\tauCkquok\subseteq\tauCkquol$ by Lemma \ref{maxquotmonotolem}. 
The change of base $k$ to $\om$ for $N$ is defined by \[N[k\mapsto\om]:=\al.\]
\end{defi}

These preparations allow us to quite canonically generalize Goodstein's theorem \cite{Goodstein1944} and Cichon's independence proof in \cite{Cichon1983} to obtain independence
of the theory $\mathcal{T}$. 
The argumentation is as follows: Given a starting base $k\in[2,\om)$ and a natural number $N\in\N$, 
by Theorem \ref{maxquotisomtheo} we obtain a {\it unique} $\al\in\tauCkquok$ such that \[N=\Ck(\al)=:N_1.\] 
The uniqueness of $\al$ such that $\Ck(\al)=N$ in $\tauCkquok$ allows us to understand $\al$ as {\it the base-$k$ representation of $N$}, 
and gives rise to the base transformation $N[k\mapsto l]:=\Cl(\al)$, so that the Goodstein operation of incrementing the base of representation and subsequent
subtraction of $1$ can be defined by \[N_1[k\mapsto k+1]-1=\Cke(\al)-1=\Pke\Cke\al=\Cke\Pke\al=:N_2,\] 
where we used Lemma \ref{tricklem}. Note that $\Pke\al\in\tauCkquoke$ by Corollary \ref{maxquotclscor} and Lemma \ref{maxquotmonotolem}. 
Iterating the procedure yields \[N_2[k+1\mapsto k+2]-1=\Ckz(\Pke\al)-1=\Pkz\Ckz(\Pke\al)=\Ckz\Pkz(\Pke\al)=:N_3,\]
so that we obtain the generalized Goodstein sequence \[N_{l+1}=\Ckl(\Pkl\ldots\Pke\al),\]
where $\Pkl\ldots\Pke\al$ is the unique preimage of $N_{l+1}$ under $\Cklstar$.
Termination of the generalized Goodstein sequence is therefore expressed by \[\exists l\:N_{l+1}=0,\]
and noting that $\Ck\al=0$ if and only if $\al=0$, we may reformulate the generalized Goodstein principle as follows:
\begin{equation}\label{reformulation}\forall k\in[2,\om)\:\forall\al\in\tauCkquok\:\exists l\:\:\Pkl\ldots\Pke\al=0.\end{equation}
Thus, transfinite induction up to $\tau$ proves this generalized Goodstein principle, since $\Pj\be<\be$ for any $j>0$ and any $\be\in(0,\tau)$.
Independence then follows, once we show that the function $\hk:\tau\to\N$ defined by 
\begin{equation}\label{hkdefieq}\hk(\al):=k + \min\{l\mid\Pkl\ldots\Pke\al=0\}\end{equation}
can easily be expressed in terms of the Hardy function $H_\al$, as in \cite{Cichon1983} for the original Goodstein principle. Defining the Hardy hierarchy along the initial segment $\tau$ 
in the same way as in \cite{Cichon1983} by 
\begin{equation}\label{Hardydefi} H_0(x):=x,\:H_{\al+1}(x):=H_\al(x+1), \mbox{ and }H_\la(x):=H_{\la[x]}(x),\end{equation}
we obtain by straightforward $<$-induction on $\al$, simultaneously in $k$:
\begin{equation}\label{hkHardyeq} \hk(\al)=H_\al(k),\end{equation}
noting that for $\al\in\Lim$ we have $\Pke(\al)=\Pke(\al[k])$, so that the i.h.\ for $\al[k]$ yields \[\hk(\al)=\hk(\al[k])=H_{\al[k]}(k)=H_\al(k),\]
and that in the successor case $\al=\be+1$, the situation $\be=0$ is trivial, while for $\be>0$, since $\Pke(\be+1)=\be$, we have 
\[\hk(\be+1)=\hke(\be)=k+1+\min\{l\mid \Pkel\ldots\Pkz\be=0\}=H_\be(k+1)=H_{\be+1}(k),\]
where we used the i.h.\ for $\be$ at $k+1$.

Since the function defined in (\ref{hkdefieq}) is expressed in terms of the Hardy hierarchy as in (\ref{hkHardyeq}), if 
the generalized Goodstein principle were provable in $\mathcal{T}$, we would obtain a proof in $\mathcal{T}$ of the totality of $H_\tau$ as defined by diagonalization along 
the master sequence: \[H_\tau(k):=H_{\tau_k}(k),\] 
using that $(\tau_n)_{n<\om}$ is completely contained in each $\tauCkquok$.
Independence of the resulting abstract Goodstein principle thus
follows from the unprovability of totality of the Hardy function $H_\tau$ in $\mathcal{T}$, which we assumed to be derivable from a proper ordinal analysis of $\mathcal{T}$ 
in the usual way. We thus obtain

\begin{cor} The Goodstein process defined according to base transformation as in Definition \ref{basechangedefi} always terminates; this statement, however, exceeds
the proof-strength of the theory $\mathcal{T}$. 
\end{cor}

\section{Conclusion and outlook}

The concept of quotients allows for a most
direct approach that enables a natural connection between complexity of ordinal notations and fast growing number theoretic functions. 
Moreover, we obtain an abstract notion of base-$k$ representation of natural numbers consisting of single ordinal numbers, which yields a meaningful direct correspondence
between ordinal and natural numbers via the mappings $\Ckstar$, which, depending on the underlying ordinal $\tau$, collapse the essence of ordinal notation systems into
the natural numbers. We consider this to be an insight that contributes to the quest for {\it natural well-orderings}, the interplay of large cardinal numbers, computable ordinals, and the
natural numbers. 


\section*{Acknowledgement}
I would like to thank Professor Andreas Weiermann for numerous insightful and motivating discussions on ordinal notation systems and Goodstein principles,
and for the suggestion to verify the characterization in terms of maximality in an earlier, concrete formulation of quotients based on restricted iteration of ordinal operations
(the $\imc$-quotients as mentioned in the introduction).


{\small
\begin{thebibliography}{99}
\bibitem{AFDWW20} T.\ Arai, D.\ Fern\'andez-Duque, S.\ S.\ Wainer, A.\ Weiermann:
    \textit{Predicatively unprovable termination of the Ackermannian Goodstein process.}
    Proceedings of the American Mathematical Society 148 (2020) 3567--3582. 
\bibitem{AWW21} T.\ Arai, S.\ S.\ Wainer, A.\ Weiermann:   
    \textit{Goodstein sequences based on a parametrized Ackermann–P\'eter function.} 
    The Bulletin of Symbolic Logic 27 (2021) 168--186.
\bibitem{BCW94} W.\ Buchholz, A.\ Cichon, A.\ Weiermann:
     \textit{A Uniform Approach to Fundamental Sequences and Hierarchies.}
     Mathematical Logic Quarterly 40 (1994) 273--286.
\bibitem{Cichon1983} E.\ A.\ Cichon:
    \textit{A short proof of two recently discovered independence results using recursion theoretic methods.}
    Proceedings of the American Mathematical Society 87 (1983) 704--706.
\bibitem{FDW24b} D.\ Fern\'andez-Duque and A.\ Weiermann:
    \textit{A walk with Goodstein.}
    The Bulletin of Symbolic Logic 30 (2024) 1--19.
\bibitem{FDW25} D.\ Fern\'andez-Duque and A.\ Weiermann:
    \textit{Fast Goodstein walks.}
    Bulletin of the London Mathematical Society 57 (2025) 510--533.
\bibitem{Gentzen1936} G. Gentzen:
    \textit{Die Widerspruchsfreiheit der reinen Zahlentheorie.}
    Mathematische Annalen 112 (1936) 493--565.
\bibitem{Gentzen1938} G. Gentzen:
    \textit{Neue Fassung des Widerspruchsfreiheitsbeweises f\"ur die reine Zahlentheorie.}
    Forschung in Logik und Grundlegung der exakten Wissenschaften.
    Heft 4, 19--44. Hirzel, Leipzig 1938.
\bibitem{Goodstein1944} R.\ L.\ Goodstein:
    \textit{On the restricted ordinal theorem.}
    The Journal of Symbolic Logic 9 (1944) 33--41.
\bibitem{Kreisel1952} G.\ Kreisel: 
    \textit{On the interpretation of non-finitist proofs II.} 
    The Journal of Symbolic Logic 17 (1952) 43--58.
\bibitem{We17} A.\ Weiermann:
    \textit{Ackermannian Goodstein principles for first order Peano arithmetic.}
    S.-D.\ Friedman, D.\ Raghavan, Y.\ Yang (eds.): {\it Sets and Computations}, Lecture Notes Series Vol.\ 33, 
    Institute for Mathematical Sciences, National University of Singapore, World Scientific Publishing Company (2017) 157--181. 
\bibitem{WW13} A.\ Weiermann and G.\ Wilken:
    \textit{Goodstein sequences for prominent ordinals up to the ordinal of $\pioneonecanod$.}
    Annals of Pure and Applied Logic 164 (2013) 1493--1506.
\bibitem{W26} G.\ Wilken:
   \textit{Fundamental sequences based on localization.}
    Annals of Pure and Applied Logic 177 (2026) 1--43.
\bibitem{W} G.\ Wilken:
   \textit{Generalizing Goodstein's theorem and Cichon's independence proof.}
   Available at: https://arxiv.org/abs/2511.04526.
\bibitem{W_this_on_arXiv} G.\ Wilken:
   \textit{Uniform generalization of Goodstein's theorem and Cichon's independence proof.}
   Link to the present article: https://arxiv.org/abs/2609.24214.
\end{thebibliography}}
\end{document}